\documentclass{amsart}
\usepackage[polish,english]{babel}
\usepackage{a4wide,amsmath,amsfonts,amssymb,amsthm,mathtools}
\usepackage{amsfonts}
\usepackage{mathtools,amssymb}
\usepackage{amsmath}
\usepackage{amsthm}
\usepackage{graphicx} 
\usepackage{wrapfig}
\usepackage{xcolor}
\usepackage{color}
\usepackage[utf8]{inputenc}
\usepackage[margin=1.5cm]{geometry}
\usepackage{enumerate}
\usepackage{hyperref}
\usepackage{bm}

\usepackage{a4wide,amsmath,amsfonts,amssymb,amsthm,mathtools}
\usepackage[all]{xy}
\usepackage{cancel}
\usepackage[normalem]{ulem}

\hypersetup{
    colorlinks=true,
    linkcolor=blue,
    filecolor=magenta,      
    urlcolor=cyan,
    pdftitle={Overleaf Example},
    pdfpagemode=FullScreen,
    }

\newcommand{\R}{\mathbb{R}}
\newcommand{\ve}{\varepsilon}
\newcommand{\on}{\operatorname}

\newcommand{\N}{\mathbb{N}}

\newcommand{\K}{\mathcal{K}}
\newcommand{\I}{\mathcal{I}}

\newcommand{\F}{\mathcal{F}}

\newcommand{\diam}{\text{diam }}
\newcommand{\Lip}{\text{Lip}}

\renewcommand{\i}{\mathbf{i}}

\newtheorem{theorem}{Theorem}[section]

\newtheorem{lemma}{Lemma}[section]

\newtheorem{remark}{Remark}[section]
\newtheorem{corollary}{Corollary}[section]

\newtheorem{example}{Example}[section]

\title{A non-logarithmic approach to the rate of convergence of the deterministic chaos game}

\author{Krzysztof Caban$^{1}$}
\thanks{$^{1}$Institute of Mathematics, University of Gdańsk, Wita Stwosza 57
80-308 Gdańsk (Poland);
ORCID 0009-0009-1608-2761
Email: krzysztof.caban@phdstud.ug.edu.pl}
\author{Filip Strobin$^{2}$}
\thanks{$^{2}$Institute of Mathematics, University of Gdańsk, Wita Stwosza 57
80-308 Gdańsk (Poland);
ORCID 0000-0002-8671-9053
Emial: filip.strobin@ug.edu.pl}
\thanks{Coresponding author: Krzysztof Caban}


\begin{document}

\begin{abstract}
    The aim of this paper is to provide a different perspective in the study of the rate of convergence of the chaos game algorithm to the attractor of an iterated function system. We prove that for any function $\psi$ with $\lim\limits_{\ve\to 0}\psi(\ve)=\infty$, a typical (in the sense of the Baire category) driver yields a rate of recovery comparable to $\psi$. This result extends the main theorem from Leśniak et al. (Rev. Real Acad. Cienc. Exactas Fis. Nat. Ser. A-Mat. 118, 157, 2024). Moreover, thanks to the change of perspective, we are able to prove that a typical driver gives arbitrarily slow rate of recovery.
\end{abstract}

\maketitle

\section{Introduction}
There are two main methods for generating attractors of iterated  function systems (IFSs), namely the algorithm based on the famous Hutchinson theorem (see \cite{Bar}) and the chaos game algorithm. The latter is a probabilistic algorithm - starting with a point $x_0\in X$, we randomly select a function from an IFS $\F=\{f_1,...,f_K\}$ and apply it to $x_0$ obtaining $x_1$. In the next iterations of the algorithm, we repeat the process but with $x_n$ in a place of $x_{n-1}$. It has been proved that the chaos game works in various different settings. For example if an IFS has an attractor $A$, then with probability $1$, the chaos game recovers the attractor, that is
\[
    A = \bigcap\limits_{n=1}^\infty \overline{\{x_i\}_{i=n}^\infty}
\]
(for a detailed explanation, we refer to \cite{BLR} and \cite{Barientos}). We can also consider the version of the algorithm in which, instead of randomly selected functions, we use a \emph{concrete} driver $\i=(i_n)$ to construct the orbit $\{x_n\}_{n\geq 0}$ (in this context, we speak about \emph{deterministic} chaos game, in the contrast to the classical one, which can be called as the \emph{probabilistic} chaos game). It is known that for a wide class of IFSs, this version of the chaos game recovers the attractor if the driver is disjunctive (see \cite{Les} and \cite{Barientos}). It is important to note that those results do not say anything about how "fast" the constructed orbit recovers $A$, i.e., how fast the orbit reaches certain resolution.

The rate of convergence of the probabilistic chaos game has been considered in a recent paper \cite{Barany} by B\'ar\'any et al. In parallel, the authors in \cite{LSS1} have studied whether similar results can be obtained in the case of the chaos game driven by disjunctive sequences. It turns out that the typical (in the sense of the Baire category) driver that recovers the attractor doesn't guarantee fast convergence. Later, in \cite{LSS2} it has been shown that a generic driver archives every possible exponent of the rate of recovery of an attractor.

In this paper we extend the study of the chaos game undertaken in \cite{LSS1} and \cite{LSS2}. Compared to \cite{LSS2}, we take a closer look at the behaviour of the main measure of the rate of recovery, that is, the value
$$
n_{\i,x_0}(\ve):=\min\{n\in\N:\{x_0,...,x_n\}\;\mbox{recovers the attractor with precision }\ve\},
$$
where  $\{x_n\}_{n\geq 0}$ is the orbit driven by $\i=(i_n)$ with starting point $x_0\in X$ (see Section 2 for precise definitions).\\
In \cite{Barany}, \cite{LSS1} and \cite{LSS2} there are studied the limit behaviour of the fraction
$$
\frac{\ln(n_{\i,x_0}(\ve))}{\ln(\frac{1}{\ve})}\;\;\;\mbox{for}\;\;\;\ve\to 0,
$$
which is close to comparing $n_{\i,x_0}(\ve)$ with the exponent maps $\psi_z(\ve)=\big(\frac{1}{\ve}\big)^z$.\\
Our idea is to "strip" the logarithms. On one hand, this approach simplifies a bit the computations, and on the other, allows to consider different maps than $\psi_z$. Our main result (Theorem \ref{thm:MAIN}) states that for any map $\psi$ with $\lim\limits_{\ve\to 0}\psi(\ve)=\infty$, the rate of recovery of an attractor of a typical driver is somehow comparable to $\psi$. The proof of Theorem \ref{thm:MAIN} (stated in Section 4 and preceded by auxiliary lemmas stated in Section 3) is an adaptation of the proof of \cite[Thm. 1.2]{LSS2} but with a few necessary modifications. In a later part, we prove that the main result from \cite{LSS2} can be deduced from our Theorem \ref{thm:MAIN}, but in some sense, our result is stronger even in the context of exponent maps (see Example \ref{ex:ext} in Section 4).  Additionally, as the result of a change of the perspective, we are able to show (see Corollary \ref{cor: MAIN implies LSS2} in Section 4) that a typical driver yields arbitrarily slow convergence. It seems that this result can not be derived from the results in \cite{LSS2}. At the end, in Section 5, we provide a brief analysis of drivers which guarantee fast rate of convergence, namely, Champernowne (see \cite{Cal}) and de Brujin ones (see \cite{All}).

\section{Preliminaries}

Throughout this section, by $X$ we denote a metric space with a metric $d$. By $c_0^+$ we denote the family of all real valued positive sequences tending to 0.

By an \emph{iterated function system} (IFS) we mean a finite family $\F=\{f_i:\ i\in I\}$, where $I$ is a finite index set and $f_i:X\longrightarrow X$ are continuous selfmaps of $X$. If each $f_i$ is a Banach contraction (that is, the Lipschitz constant $\Lip (f_i)<1$), then we say that $\F$ is a \emph{Banach contractive IFS}, and we denote $L:=\max\{\Lip(f_i):\ i\in I\}$. Furthermore, we define the \emph{Hutchinson operator}
\[
    \F(K):=\bigcup\limits_{i\in I} f_i(K),\ K\in\K(X),
\]
where $\K(X)$ is the family of nonempty compact subsets of $X$ endowed with the Hausdorff metric $d_H$.

An \emph{attractor} of an IFS $\F$ is a nonempty compact set $A\in \K(X)$ such that $\F(A)=A$ and $\F^{(n)}(K)\to A$ with respect to $d_H$ for any $K\in \K(X)$. The classical Hutchinson theorem (see \cite{Bar}) states that any Banach contractive IFS on a complete metric space admits an attractor.

The \emph{chaos game} is one of classical algorithms generating an attractor $A$ of an IFS $\F=\{f_1, \dots, f_K\}$. For a fixed nondegenerate probability vector $\mathbf{p}=(p_1,\dots, p_K)$ the algorithm is defined as follows:
\begin{itemize}
    \item choose a starting point $x_0\in X$;
    \item choose $i_1$ randomly from $\{1,\dots,K\}$ according to probabilities $\mathbf{p}$;
    \item define $x_1:=f_{i_1}(x_0)$;
    \item choose $i_2$ randomly from $\{1,\dots,K\}$ according to probabilities $\mathbf{p}$;
    \item define $x_2:=f_{i_2}(x_1)=f_{i_2}\circ f_{i_1}(x_0)$;
    \item ...and so on...
\end{itemize}
One can show that if an IFS admits an attractor, then with probability one (w.r.t. the product probability on the space $I^\infty$ of all sequences of elements from $I$), for any $\varepsilon>0$ there exists $M\in\N$ such that $d_H(A, \overline{\{x_n\}_{n=M}^\infty})\leqslant\varepsilon$ (see \cite{Mar} or the next part of this section for elementary explanation for contractive IFSs, and \cite{Les} and \cite{Barientos} for a general case). We can look at the chaos game from a different perspective, namely, we can consider a deterministic approach. In this approach we do not randomly generate a sequence $\i=(i_k)_{k=1}^\infty$, but we use a fixed driver $\i$ to generate the orbit.

Let $\F=\{f_i:i\in I\}$ be an IFS on a metric space $X$, where $I=\{1,...,K\}$ for some $K\in\N$.\\
For any sequence $\i\in I^\infty$ and $x_0\in X$, we define the sequence $(x_k)_{k=0}^\infty$ in the following way:
\[
    x_k=f_{i_k}(x_{k-1})=f_{i_k}\circ \cdots f_{i_1}(x_0),\ k\in\N,
\]
and we call it as the \emph{orbit of} $\F$ along a \emph{driver} $\i$ with \emph{starting point} $x_0\in X$.\\ 
In an analogous way we define the finite orbit $(x_k)_{k=0}^m$ for a finite driver $\i\in I^m$.\\
The space of infinite drivers $I^\infty$ is endowed with the Tikhonov product topology, which is compact and metrizable.

We say, that a finite orbit $(x_k)_{k=0}^n$ driven by $\i$ \emph{recovers an attractor $A$ with precision} $\varepsilon>0$, if
\[
    d_H(A, \{x_k: k=k_0,\dots, n\})\leqslant \varepsilon, \text{ for some }k_0=0,1,\dots,n.
\]
The rate at which an orbit $(x_k)_{k=0}^\infty$ recovers $A$ can be determined by the following quantity
\[
    n_{\i, x_0}(\varepsilon):=\min\big\{ n\geqslant 0:\ d_H(A, \{x_k: k=k_0,\dots, n\})\leqslant \varepsilon \text{ for some }k_0=0,1,\dots,n \big\}.
\]
If $\F$ consists of nonexpansive maps, the definition of $n_{\i, x_0}$ simplifies to
\[
    n_{\i, x_0}(\varepsilon)=\min\bigg\{ n\geqslant 0:\ A\subseteq \bigcup_{k=0}^n B(x_k, \varepsilon) \bigg\},
\]
where $B(y,r)$ denotes the closed ball cenetred at $y$ with a radius $r$ ($r$-ball, for short).\\
It is known that, if $\F$ is Banach contractive and $\i$ is disjunctive, i.e. every finite word appears in $\i$ (see \cite{Cal}), then $A$ is recovered with arbitrary precision for any starting point, that is $n_{\i, x_0}(\varepsilon)<\infty$ for any $\varepsilon>0$ and any $x_0\in X$ (see \cite{BaVi}). This gives probably the simplest expanation for ''probabilisitic'' chaos game for contractive IFSs, since with probability one, a disjunctive driver is chosen (see the mentioned paper \cite{Mar}).

In \cite{LSS1} and \cite{LSS2} the authors studied the asymptotics of the maps
\[
    \varepsilon\mapsto \frac{\ln(n_{\i,x_0}(\varepsilon))}{\ln(\frac{1}{\varepsilon})},
\]
to describe the rate of convergence of the deterministic chaos game (see also \cite{Barany} for probabilistic approach to this problem).

For a sequence $\i=(i_k)\in I^\infty$, define
\[
    \mathbf{n}_\i(m):=\min\{ n\in\N:\ (i_1,\dots, i_k)\text{ contains all finite words of length } m \}.
\]
It turns out that the properties of $\i$, related to rate of convergence, depend on the behavior of $\mathbf{n}_\i(m)$. First of all, it is clear that a sequence $\i\in I^\infty$ is \emph{disjunctive} iff $\mathbf{n}_\i(m)<\infty$ for all $m\in\N$. As showed in \cite[Thm. 23]{LSS1}, \cite[Thm. 20]{LSS1} and \cite[Thm. 1.2]{LSS2}, a disjunctive sequence can yield arbitrarily slow convergence. This motivates the definition \cite[Def. 6, Rem. 7]{LSS1} of \emph{asymptotically Champernowne} sequences which guarantee relatively fast convergence. We say that $\i\in I^\infty$ where $I=\{1,...,K\}$, is \emph{asymptotically Champernowne}, if
\[
    \limsup\limits_{m\to\infty} \frac{\ln(\mathbf{n}_\i(m))}{m}\leqslant \ln(K).
\]
\begin{theorem}\label{thm: as Champ fast conv}\cite[Thm. 17]{LSS1}
Assume that $\F=\{f_1,...,f_K\}$ is a Banach contractive IFS on a complete metric space $X$ and $L:=\max\{\on{Lip}(f_1),...,\on{Lip}(f_k)\}$. Then
    the rate of convergence of the chaos game driven by an asymptotically Champernowne sequence $\i$ obeys the following
    \[
        \overline{\lambda}_{\i,x_0}:=\limsup_{\varepsilon\to 0} \frac{\ln(n_{\i,x_0}(\varepsilon))}{\ln(\frac{1}{\varepsilon})} \leqslant \frac{\ln(K)}{\ln(\tfrac{1}{L})},
    \]
    where $x_0$ is any point from $X$.
\end{theorem}
Moreover (see \cite[Cor. 18]{LSS1}), if $\F$ consists of similitudes with equal scaling ratios $L<1$, and its attractor $A$ satisfies the strong separation condition, then for an asymptotically Champernowne sequence $\i$, we have
\[
    \lim_{\varepsilon\to 0} \frac{\ln(n_{\i,x_0}(\varepsilon))}{\ln(\frac{1}{\varepsilon})} = D(A),
\]
where $D(A)$ is the box (or fractal) dimension of $A$.
In fact, the same assertion holds for such IFSs on Euclidean spaces which satisfy the Open Set Condition.\\
Recall that the \emph{lower box dimension} of a nonempty compact set $K\subset X$ is the quantity
\begin{equation}\label{eq: low box def}
    \underline{D}(K):=\liminf\limits_{\varepsilon\to0} \frac{N(\varepsilon)}{\ln(\frac{1}{\varepsilon})},
\end{equation}
where $N(\varepsilon)$ is the minimum number of closed $\varepsilon$-balls needed to cover the set $K$. If the limit in (\ref{eq: low box def}) exists, then it is called as the box dimension of $K$.\\
Now we recall the definition of a spectrum of exponents of rate of convergence $\lambda(\i)$, for $i\in I^\infty$, from \cite{LSS2}.\\
If $z\neq\infty$, then $z\in\lambda(\i)$ if there exists a sequence $(\varepsilon_k)\in c_0^+$ such that for all $R>0$,
\[
    \lim\limits_{k\to\infty} \sup\limits_{d(x_0, A)\leqslant R} \left| \frac{\ln(n_{\i,x_0}(\varepsilon_k))}{\ln(\frac{1}{\varepsilon_k})} - z \right| = 0.
\]
Additionally, $\infty\in \lambda(\i)$ if there exists a sequence $(\varepsilon_k)\in c_0^+$ such that for all $R>0$, 
\[
    \lim\limits_{k\to\infty} \inf\limits_{d(x_0, A)\leqslant R} \frac{\ln(n_{\i,x_0}(\varepsilon_k))}{\ln(\frac{1}{\varepsilon_k})} = \infty.
\]
\begin{lemma}\label{lem: spectrum is closed}\cite[Lem. 3.3]{LSS2}\\
    For every driver $\i\in I^\infty$, the set $\lambda(\i)$ is closed.
\end{lemma}

The following theorem is the main result of \cite{LSS2}

\begin{theorem}\label{thm: LSS2 MAIN}\cite[Thm. 1.2]{LSS2}\\
    Let $\F=\{f_i:i\in I\}$ be a Banach contractive IFS on a complete metric space with infinite attractor $A$. Let $\I$ be the set of all disjunctive drivers $\i\in I^\infty$ such that $\lambda(\i)=[\underline{D}(A), \infty]$. Then $\I$ is residual in $I^\infty$.
\end{theorem}
\begin{remark}\label{rem: LSS2 MAIN}\emph{
The essential part of the proof is the following fact: for every $z\in (\underline{D}(A), \infty)$, the set 
\begin{equation}\label{eq: LSS2 MAIN}
\mathcal{I}_z:=\{\i\in I^\infty:z\in\lambda(\i)\}
\end{equation}
is residual.\\
Then we use Lemma \ref{lem: spectrum is closed} and the fact that the family $\mathcal{D}$ of disjunctive drivers is residual, to see that $\I=\mathcal{D}\cap\bigcap_{z\in Z}\I_z$ for any countable and dense set $Z\subset (\underline{D}(A), \infty)$.
}\end{remark}
Note that it is well known that the family of disjunctive drivers if residual (see \cite{Cal}), so a nontrivial part in the above result is residuality of the family of drivers so that  $\lambda(\i)=[\underline{D}(A), \infty]$. 

Finally, observe that the interval $[\underline{D}(A), \infty]$ in Theorem \ref{thm: LSS2 MAIN} is the best possible. Indeed, for any $\i\in I^\infty$ and $x_0\in X$, we have
\[
    \underline{\lambda}_{\i,x_0}:=\liminf\limits_{\varepsilon\to\infty} \frac{\ln(n_{\i,x_0}(\varepsilon))}{\ln(\frac{1}{\varepsilon})} \geqslant \underline{D}(A),
\]
since for any $\varepsilon>0$, we simply have
\begin{equation}\label{eq:key}
    n_{\i,x_0}(\varepsilon) + 1\geqslant N(\varepsilon).
\end{equation}

\section{Auxiliary results}
In this section we give some auxiliary results which will be needed in the main part. We  assume that $\F=\{f_i:i\in I\}$, where $I=\{1,...,K\}$, is a Banach contractive IFS on a complete metric space $X$ and $A$ is its attractor, and $L:=\max\{\on{Lip}(f_i):i\in I\}$.\\
The first lemma (from \cite{LSS2}) gives a characterization of lower dimension (cf. Theorem 1.1 in \cite[Chap. V]{Bar}).
\begin{lemma}\label{lem: lower box dimension sequentially}\cite[Lem. 3.2]{LSS2}
    Set $b_m:= ar^m$ for $m\in\N$, where $a>0$ and $r\in(0,1)$. Then for every compact set $K\subseteq X$, there exists an increasing sequence $(m_k)$ of natural numbers such that
    \[
        \underline{D}(K)=\lim\limits_{k\to\infty}\frac{\ln (N(b_{m_k}))}{\ln\left( \frac{1}{b_{m_k}} \right)}.
    \]
\end{lemma}
The second lemma (from \cite{LSS1}) gives a way to construct finite drivers allowing good approximations of attractors.
\begin{lemma}\label{lem: finite driver approx}\cite[Lem. 19]{LSS1}\\
    For every $m\in \N$, there exists $\sigma_m\in I^{m\cdot N(C_m)}$ such that for every $x_0\in X$ with $d(x_0, A)\leqslant 1$, we have
    \[
        A \subseteq \bigcup\limits_{k=0}^{m\cdot N(C_m)} B(x_k, 3C_m),
    \]
    where $(x_k)_{k=0}^{m\cdot N(C_m)}$ is the finite orbit driven by $\sigma_m$ and $C_m:=L^m(\on{diam}A+1)$.
\end{lemma}

\begin{remark}\label{rem: new Lemma finite driver approx}\emph{
    If we slightly modify the proof of \cite[Lem. 19]{LSS1}, we can obtain a more general statement, namely:}

    For every $d>0$ and every $m\in\N$, there exists $\sigma(d,m)\in I^{m\cdot N(d)}$  such that for every $x_0\in X$, we have
    \[
        A \subseteq \bigcup\limits_{k=0}^{m\cdot N(d)} B(x_k, 2d + L^m(\diam A + d(x_0, A))),
    \]
    where $(x_k)_{k=0}^{m\cdot N(d)}$ is the finite orbit driven by $\sigma(d,m)$.
\end{remark}

We end this section with two technical lemmas useful in later proofs. The first one seems to be folklore but we give its proof for the reader's convenience.

\begin{lemma}\label{lem: convergence of logarithms}
    Assume that $(a_k),(b_k)\subseteq\R$ are such that $a_k,b_k\to\infty$ as $k\to\infty$, $b_k>1$ and $a_k>0$ for all $k\in\N$. Then
    \begin{equation}\label{eq: lemma implication}
        \lim\limits_{k\to \infty} \frac{a_k}{b_k} = 1\quad \implies \quad  \lim\limits_{k\to \infty} \frac{\ln(a_k)}{\ln(b_k)} = 1.
    \end{equation}
\end{lemma}
\begin{proof}
    Assume that 
    \[
        \lim\limits_{k\to \infty} \frac{a_k}{b_k} = 1.
    \]
    Fix any $0<\delta<1$. Then there exists $k_0\in \N$ such that for every $k\geqslant k_0$, we have
    \[
        b_k(1-\delta)\leqslant a_k \leqslant b_k(1+\delta).
    \]
    Applying the natural logarithm to the above inequality, we obtain
    \[
        \ln(b_k) + \ln(1-\delta)\leqslant\ln(a_k) \leqslant \ln(b_k) + \ln(1+\delta)
    \]
    \[
        1 + \frac{\ln(1-\delta)}{\ln(b_k)} \leqslant \frac{\ln(a_k)}{\ln(b_k)} \leqslant 1 + \frac{\ln(1+\delta)}{\ln(b_k)}.
    \]
    Now since
    \[
        \lim\limits_{k\to\infty}\frac{\ln(1\pm\delta)}{\ln(b_k)}= 0,
    \]
    we have that
    \[
        \lim\limits_{k\to\infty} \frac{\ln(a_k)}{\ln(b_k)}= 1.
    \]
\end{proof}

\begin{lemma}\label{lem: limsup limlinf equiv}
    Let $\{(a_k^x): x\in C\}$ be a family of real valued sequences indexed by some set $C$, and let $b\in\R$. The following conditions are equivalent
    \begin{enumerate}[(i)]
        \item $\lim\limits_{k\to\infty} \left(\sup\limits_{x\in C} |a_k^x - b|\right) = 0$; 
        \item $\limsup\limits_{k\to\infty} \left(\sup\limits_{x\in C} a_k^x\right)\leqslant b\;\;$ and $\;\; \liminf\limits_{k\to\infty} \left(\inf\limits_{x\in C} a_k^x\right)\geqslant b$
    \end{enumerate}
\end{lemma}
\begin{proof}
    We begin the proof by noting that
    \begin{equation}\label{eq: liminf limsup equation}
        \liminf\limits_{k\to\infty}\left(\inf\limits_{x\in C} a_k^x\right) \leqslant \limsup\limits_{k\to\infty}\left(\inf\limits_{x\in C} a_k^x \right)\leqslant \limsup\limits_{k\to\infty}\left(\sup\limits_{x\in C} a_k^x\right).
    \end{equation}
    and
    \begin{equation}\label{eq: liminf liminf equation}
        \liminf\limits_{k\to\infty}\left(\inf\limits_{x\in C} a_k^x\right) \leqslant \liminf\limits_{k\to\infty}\left(\sup\limits_{x\in C} a_k^x \right)\leqslant \limsup\limits_{k\to\infty}\left(\sup\limits_{x\in C} a_k^x\right),
    \end{equation}
    Now assume $(ii)$. By (\ref{eq: liminf limsup equation}) and (\ref{eq: liminf liminf equation}), we have that
    \[
        \lim\limits_{k\to\infty}\left(\inf\limits_{x\in C} a_k^x\right) = b\ \text{ and }\ \lim\limits_{k\to\infty}\left(\sup\limits_{x\in C} a_k^x\right) = b.
    \]
    Since $\lim\limits_{k\to\infty}\left(\inf\limits_{x\in C} a_k^x \right)= b$, we have
    \begin{equation}\label{eq: all akx bigger}
        \forall_{\varepsilon>0}\ \exists_{k_1\in \N}\ \forall_{k\geqslant k_1}\ \forall_{x\in C}\ b-\varepsilon\leqslant a_k^x.
    \end{equation}
    Since $\lim\limits_{k\to\infty}\left(\sup\limits_{x\in C} a_k^x \right)= b$, we have
    \begin{equation}\label{eq: all akx lower}
        \forall_{\varepsilon>0}\ \exists_{k_2\in \N}\ \forall_{k\geqslant k_2}\ \forall_{x\in C}\ a_k^x\leqslant b+\varepsilon.
    \end{equation}
    Combining (\ref{eq: all akx bigger}) and (\ref{eq: all akx lower}), we get
    \begin{equation}\label{eq: condition i}
        \forall_{\varepsilon>0}\ \exists_{k_0\in \N}\ \forall_{k\geqslant k_0}\ \forall_{x\in C}\ b-\varepsilon\leqslant a_k^x \leqslant b+\varepsilon,
    \end{equation}
    equivalently
    \[
        \lim\limits_{k\to\infty} \left(\sup_{x\in C} |a_k^x - b|\right) = 0.
    \]

    Now assume $(i)$. Then we have condition (\ref{eq: condition i}) and hence for any $\varepsilon>0$ there exists $k_0\in\N$ such that for all $k\geqslant k_0$
    \[
        b-\varepsilon \leqslant \inf\limits_{x\in C} a_k^x \leqslant \sup\limits_{x\in C} a_k^x \leqslant b+\varepsilon.
    \]
    This establishes that
    \[
      \ \lim\limits_{k\to\infty}\left(\sup\limits_{x\in C} a_k^x\right) = b\;\;  \text{ and }\;\;\lim\limits_{k\to\infty}\left(\inf\limits_{x\in C} a_k^x\right) = b,
    \]
    and thus we get $(ii)$.
\end{proof}

\section{Main results}
As earlier, we  assume that $\F=\{f_i:i\in I\}$, where $I=\{1,...,K\}$, is a Banach contractive IFS on a complete metric space $X$ and $A$ is its attractor, and $L:=\max\{\on{Lip}(f_i):i\in I\}$.\\
Again, for any $m\in \N$, define
$$
    C_m:= L^m (\diam A + 1).
$$

\begin{theorem}\label{thm:MAIN}
    Let $\psi:(0,\infty)\longrightarrow(0,\infty)$ be any function such that $\lim\limits_{\varepsilon\to 0}\psi(\varepsilon)=\infty$.\\
If 
    \begin{equation}\label{eq: main assumptions}
        \liminf_{n\to\infty}n\frac{N(C_{n})}{\psi(3C_{n})}=0,
    \end{equation}
    then the set
   \begin{equation}\label{eq:mainset}
        \I_\psi:=\big\{\i\in I^\infty:\ \exists_{(\varepsilon_k)\in c_0^+}\ \forall_{R>0}\ \sup_{d(x_0,A)\leqslant R}\left|\frac{n_{\i,x_0}(\varepsilon_k)}{\psi(\varepsilon_k)}-1\right|\to0\big\}
    \end{equation}
    is residual in $I^\infty$.
\end{theorem}
\begin{remark}\emph{
We can formulate the above assertion in a manner similar to that from Theorem \ref{thm: LSS2 MAIN} (cf. Remark \ref{rem: LSS2 MAIN}). Namely, let $\Phi$ be the family of all maps $\psi:(0,\infty)\to(0,\infty)$ with $\lim\limits_{\varepsilon\to 0}\psi(\varepsilon)=\infty$. Then for every $\i\in I^\infty$, let $\Lambda(\i)$ be the family of all $\psi\in\Phi$ so that (\ref{eq:mainset}) holds. Then Theorem \ref{thm:MAIN} states that for any $\psi\in\Phi$ which satisfies (\ref{eq: main assumptions}), the family 
$$
\{\i\in I^\infty:\psi\in\Lambda(\i)\}
$$
is residual in $I^\infty$.\\
As we remarked, the proof of Theorem \ref{thm:MAIN} will be an adaptation of the proof of Theorem \ref{thm: LSS2 MAIN}.}
\end{remark}
\begin{proof}
    Fix $i_*$ and a corresponding fixed point $x_*$ of $f_{i_*}$ such that $A$ has a different accumulation point than $x_*$. Then there exists $\delta>0$ such that $A\setminus B(x_*,\delta)$ is infinite. For the rest of the proof let $N_\delta(\varepsilon)$ denote the minimum number of $\varepsilon-$balls needed to cover the set $A\setminus B(x_*,\delta)$.

    We will define by induction two increasing sequences $(m_k),(p_k)\subseteq\N$. First, let $(n_l)\subseteq \N$ be such that 
    \begin{enumerate}[(A)]
        \item $\lim\limits_{l\to\infty} n_{l} \frac{N(C_{n_l})}{\psi(3C_{n_l})}=0.$
\end{enumerate}
        Next choose $m_1:=n_{l_1}$ such that
\begin{enumerate}[(B)]
        \item $C_{m_1}<\frac{\delta}{2}.$
\end{enumerate}
Then put $p_1:=\lceil \psi(3C_{m_{1}}) \rceil$. Now, having constructed $m_j=n_{l_j}$ and $p_j$ for $j=1,\dots,k$, we define $m_{k+1}=n_{l_{k+1}}>n_{l_k}$ and $p_{k+1}$ so that
\begin{enumerate}[(A)]
    \setcounter{enumi}{2}
        \item $N_\delta(3 C_{m_{k+1}})>v(k)+m_1$, where $v(k):=\sum\limits_{j=1}^k(p_j + m_j N(C_{m_j}))$;
        \item $N(C_{m_{k+1}})> v(k)$;
        \item $m_{k+1}>m_k + k$;
        \item $p_{k+1}:=\lceil \psi(3C_{m_{k+1}}) \rceil$.
    \end{enumerate}
The construction can be handled as $\lim\limits_{\varepsilon\to 0}N_\delta(\varepsilon)=\lim\limits_{\ve\to 0}N(\ve)=\infty$.\\
    \textbf{Claim} We have
    \[
        \left| p_k - \psi(3 C_{m_k}) \right|\leqslant 1\;\; \text{ and } \;\;\lim\limits_{k\to\infty} \frac{(m_k+1) N(C_{m_k})}{p_k} =0.
    \]
    \textbf{Proof of Claim} The first inequality follows from the definition of $p_k$. For the second one, observe that for sufficiently large $k$, by (A), where $(n_l)$ is replaced by its subsequence $(m_k)$, and by (F), we have
    \begin{align*}
        \frac{(m_k+1) N(C_{m_k})}{p_k}&\leqslant \frac{(m_k+1) N(C_{m_k})}{\psi(3C_{m_k})}= m_k \frac{N(C_{m_k})}{\psi(3C_{m_k})} + \frac{N(C_{m_k})}{\psi(3C_{m_k})}\to 0.
    \end{align*}
    Continuing the proof of Theorem \ref{thm:MAIN}, for $\i\in I^\infty$, put
    \[
        \i[k]:=(i_{v(k-1)+1},\dots, i_{v(k)})\in I^{p_k + m_k N(C_{m_k})},
    \]
    where $v(0):=0$. Also for every $n\in\N$, define
    \[
        U_n := \bigcup\limits_{k=n}^\infty \{ \i\in I^\infty:\ \i[k]=(i_*)_{p_k} \sigma_{m_k} \},
    \]
    where $(i_*)_{p_k} \sigma_{m_k}$ is the concatenation of the constant sequence $(i_*)_{p_k}$ consisting of $i_*$ and of the length $p_k$, with a finite driver $\sigma_{m_k}$ from Lemma \ref{lem: finite driver approx}.

    Following the same argument as in the proof of Theorem \ref{thm: LSS2 MAIN} one can prove that each $U_n$ is open and dense in $I^\infty$. Hence, their intersection $\bigcap_{n\in\N} U_n$ is residual in $I^\infty$.

    To prove that $\I_\psi$ is residual we will show that $\bigcap_{n\in\N} U_n\subseteq \I_\psi$. Let $i\in \bigcap_{n\in\N} U_n$. Then there is an increasing sequence $(k_l)$ of natural numbers such that for every $l\in \N$,
    \begin{equation}\label{eq: segments of i}
        \i[k_l] = (i_*)_{p_{k_l}} \sigma_{m_{k_l}}.
    \end{equation}
    To show that $\i\in\I_\psi$, based on Lemma \ref{lem: limsup limlinf equiv}, it is enough to show that for every $R>0$,
    \begin{equation}\label{eq: limsup}
        \limsup\limits_{l\to\infty}\sup\limits_{d(x_0, A)\leqslant R} \frac{ n_{\i, x_0}(\varepsilon_l)}{\psi(\varepsilon_l)}\leqslant 1    
    \end{equation}
    and
    \begin{equation}\label{eq: liminf}
        \liminf\limits_{l\to\infty}\inf\limits_{d(x_0, A)\leqslant R} \frac{n_{\i, x_0}(\varepsilon_l)}{\psi(\varepsilon_l)}\geqslant 1,
    \end{equation}
where $\varepsilon_l := 3C_{m_{k_l}}$, $l\in\N$ (we use Lemma \ref{lem: limsup limlinf equiv} for $C=\{x_0:d(x_0,A)\leq R\}$ and $a_l^{x_0}=\frac{n_{\i, x_0}(\varepsilon_l)}{\psi(\varepsilon_l)}$).

    For the rest of the proof, fix $R>0$, and $l_0$ so that $L^{v(k_{l_0}-1)}\leqslant R^{-1}$. Observe that for any $l\geqslant l_0$, $x_0\in X$ with $d(x_0, A)\leqslant R$ and any $j\in \N$, we have
    \begin{equation}\label{eq: distance leq 1}
        d(x_{v(k_l-1)+j}, A)\leqslant L^{v(k_l-1)+j}d(x_0, A)\leqslant L^{v(k_{l_0}-1)}d(x_0, A)\leqslant 1.
    \end{equation}
    Hence, taking $j=p_{k_l}$ in (\ref{eq: distance leq 1}), we can apply Lemma \ref{lem: finite driver approx} with $x_{v(k_l-1)+p_{k_l}}$ as $x_0$ and then using (\ref{eq: segments of i}) and (D) we obtain
    \[
        n_{\i, x_0}(3C_{m_{k_l}})\leqslant v(k_l-1) + p_{k_l} + m_{k_l}N(C_{m_{k_l}}) \leqslant N(C_{m_k}) + p_{k_l} + m_{k_l}N(C_{m_{k_l}}).
    \]
    Therefore, by \textbf{Claim} we have
    $$
        \limsup\limits_{l\to\infty}\sup\limits_{d(x_0, A)\leqslant R} \frac{ n_{\i, x_0}(\varepsilon_l)}{\psi(\varepsilon_l)}= \limsup\limits_{l\to\infty}\sup\limits_{d(x_0, A)\leqslant R} \frac{ n_{\i, x_0}(3C_{m_{k_l}})}{\psi(3C_{m_{k_l}})}\leq $$
        $$\leq \limsup\limits_{l\to\infty} \frac{p_{k_l} + (m_{k_l}+1)N(C_{m_{k_l}})}{\psi(3C_{m_{k_l}})}= \lim\limits_{l\to\infty}\left( \frac{p_{k_l}}{\psi(3C_{m_{k_l}})} + \frac{(m_{k_l}+1)N(C_{m_{k_l}})}{\psi(3C_{m_{k_l}})}\right) = 1.
 $$
    This gives (\ref{eq: limsup}).

   Now let $l_1\geqslant l_0$ be such that that $p_{k_l}-m_1>0$ and $3C_{m_{k_l}}<\frac{\delta}{2}$ for $l\geq l_1$. Take any $x_0\in X$ with $d(x_0,A)\leq R$. By (C),  the set $A\setminus B(x_*, \delta)$ is not contained in the union of balls $B(x_i, 3C_{m_{k_l}}),\ i=1,\dots, v(k_l-1)+m_1$. Moreover, for all $j=1,\dots, p_{k_l}$, we have that
    \[
        x_{v(k_l-1)+j} = f_{i_*}^{(j)}(x_{v(k_l-1)}).
    \]
    Hence, by (\ref{eq: distance leq 1}) and by (B), for $j=1,\dots, p_{k_l}-m_1$, choosing $y\in A$ such that $d(x_{v(k_l-1)+j}, y) = d(x_{v(k_l-1)+j}, A)$, we have
    $$
        d(x_{v(k_l-1)+m_1+j}, x_*) = d(f_{i_*}^{(m_1)}(x_{v(k_l-1)+j}), f_{i_*}^{(m_1)}(x_*))
        \leqslant L^{m_1} (d(x_{v(k_l-1)+j}, y) + d(y, x_*))
        \leqslant L^{m_1}(1+ \text{diam} A) = C_{m_1}<\frac{\delta}{2}.
    $$
    Then
    \[
        B(x_{v(k_l-1)+m_1+j}, 3C_{m_{k_l}})\subseteq B\left( x_{v(k_l-1)+m_1+j}, \frac{\delta}{2} \right)\subseteq B(x_*, \delta).
    \]
    In particular, the family of balls $B(x_i, 3C_{m_{k_l}}),\ i=0,\dots, v(k_l-1)+p_{k_l}$ does not cover $A$ and hence
    \[
        n_{\i,x_0}(3C_{m_{k_l}})\geqslant v(k_l-1) + p_{k_l}\geqslant p_{k_l}.
    \]
    Finally, by \textbf{Claim},
    \[
        \liminf\limits_{l\to\infty}\inf\limits_{d(x_0, A)\leqslant R} \frac{ n_{\i, x_0}(\varepsilon_l)}{\psi(\varepsilon_l)} = \liminf\limits_{l\to\infty}\inf\limits_{d(x_0, A)\leqslant R} \frac{ n_{\i, x_0}(3C_{m_{k_l}})}{\psi(3C_{m_{k_l}})}\geqslant \lim\limits_{l\to\infty} \frac{p_{k_l}}{\psi(3C_{m_{k_l}})} = 1,
    \]
    which gives (\ref{eq: liminf}). This completes the proof.
\end{proof}

The following remark gathers some observations on the assumptions of Theorem \ref{thm:MAIN} and details of its proof. 

\begin{remark}
    \phantom{aa}
    \begin{enumerate}
        \item \emph{If $(X,d)$ is doubling, i.e. for every $c\in(0,1)$, there exists $M_c>0$ such that for every $x\in X$ and $R>0$, it is possible to cover $B(x,R)$ with at most $M$ balls of radius $cR$ (Euclidean spaces serve as classical examples), then assumption (\ref{eq: main assumptions}) in Theorem \ref{thm:MAIN} is equivalent to: }
        
            \[
                \liminf_{n\to\infty} n \frac{N(L^{n})}{\psi(3C_{n})}=0.
            \]
           \emph{ Indeed, this follows from the inequality
            \[
                N(\varepsilon)\leqslant N(c\varepsilon)\leqslant M_c\cdot N(\varepsilon),\ \varepsilon>0.
            \]}
        \item \emph{We can formulate an alternative version of Theorem \ref{thm:MAIN}  only taking into account $x_0\in A$:}
        
        Assume that
        \[
            \liminf\limits_{n\to\infty} n \frac{N(C_{n}')}{\psi(3C_{n}')}=0,
        \]
        where $C_m':=\operatorname{diam}A\cdot L^m$ for $m\in \N$. Then the set
        \[
            \I_\psi':=\{\i\in I^\infty:\ \exists_{(\varepsilon_k)\in c_0^+}\ \sup_{x_0\in A}\left|\frac{n_{\i,x_0}(\varepsilon_k)}{\psi(\varepsilon_k)}-1\right|\to0\}
        \]
        is residual in $I^\infty$.

        \emph{The proof of this version is very similar - use a version of Lemma \ref{lem: finite driver approx} presented in Remark \ref{rem: new Lemma finite driver approx} applied to $x_0\in A$ and $d_m=C_m'$.}
        %
    \end{enumerate}
    
\end{remark}

Now we formulate two corollaries of Theorem \ref{thm:MAIN}. Firstly, we show that the set of drivers which yield arbitrarily slow convergence is residual. Next we prove that Theorem \ref{thm: LSS2 MAIN} follows from Theorem \ref{thm:MAIN}.

In the following corollary by $\ln^{(n)}$ we mean the $n$-th composition of the natural logarithm. The domian of $\ln^{(n)}$, $n\geqslant 2$, is the following set $D_n:= \{x\in\R:\ \ln^{(n-1)}(x)>0\}=(\exp^{(n-1)}(0),\infty)$ and we adopt the notation that for $x\notin D_n$ we have $\ln^{(n)}(x)=-\infty$.
\begin{corollary}\label{cor: I infty is residual}
    The set
    \[
        \I_\infty=\big\{\i\in I^\infty:\ \exists_{(\varepsilon_k)\in c_0^+}\ \forall_{n\in\N}\ \forall_{R>0} \ \inf\limits_{d(x_0, A)\leqslant R}\frac{\ln^{(n)}(n_{\i,x_0}(\varepsilon_k))}{\ln\left( \frac{1}{\varepsilon_k} \right)}\to \infty\big\}
    \]
    is residual in $I^\infty$.
\end{corollary}
\begin{proof}
    We begin the proof by observing that $\I_\infty=\I_\infty'$, where
    \begin{equation}\label{eq: I infty alternative form}
        \I_\infty'=\big\{\i\in I^\infty:\ \forall_{n\in\N}\ \exists_{(\varepsilon_k)\in c_0^+}\ \forall_{R>0}\ \inf\limits_{d(x_0, A)\leqslant R}\frac{\ln^{(n)}(n_{\i,x_0}(\varepsilon_k))}{\ln\left( \frac{1}{\varepsilon_k} \right)}\to \infty\big\}.
    \end{equation}
    Clearly $\I_\infty\subseteq \I_\infty '$, so it remains to prove that the reverse inclusion holds. Let $\i\in \I_\infty'$.

    For $n=1$, there exists a sequence $(\varepsilon^1_k)$ such that
    \[
        \inf\limits_{d(x_0, A)\leqslant 1}\frac{\ln(n_{\i,x_0}(\varepsilon^1_k))}{\ln\left( \frac{1}{\varepsilon^1_k} \right)}\to \infty,
    \]
    hence there exists $k_1\in \N$ such that $\varepsilon^1_{k_1}<1$ and
    \[
        \inf\limits_{d(x_0, A)\leqslant 1}\frac{\ln(n_{\i,x_0}(\varepsilon^1_{k_1}))}{\ln\left( \frac{1}{\varepsilon^1_{k_1}} \right)}\geqslant 1.
    \]
    Next, for $n=2$ we can find $k_2$ and $\varepsilon^2_{k_2}< \tfrac{1}{2}\varepsilon^1_{k_1}$ such that
    \[
        \inf\limits_{d(x_0, A)\leqslant 2}\frac{\ln^{(2)}(n_{\i,x_0}(\varepsilon^2_{k_2}))}{\ln\left( \frac{1}{\varepsilon^2_{k_2}} \right)}\geqslant 2.
    \]
    Continuing this process, we can find a sequence $(\varepsilon^n_{k_n})\in c_0^+$ with $\varepsilon^n_{k_n}<\frac{1}{n} \varepsilon^{n-1}_{k_{n-1}}$, where $\varepsilon^0_{k_0}:=1$, such that
    \[
        \inf\limits_{d(x_0, A)\leqslant n}\frac{\ln^{(n)}(n_{\i,x_0}(\varepsilon^n_{k_n}))}{\ln\left( \frac{1}{\varepsilon^n_{k_n}} \right)} \geqslant n.
    \]
    Define $\varepsilon_m:=\varepsilon^m_{k_m}$. We will check that for any $n_0\in\N$ and any $R>0$, we have
    \[
        \inf\limits_{d(x_0, A)\leqslant R}\frac{\ln^{(n_0)}(n_{\i,x_0}(\varepsilon_m))}{\ln\left( \frac{1}{\varepsilon_m} \right)}\to \infty.
    \]
    Fix $n_0\in\N$ and $R>0$. Then, if $m>n_0$ and $m>R$, we have
    \[
        \inf\limits_{d(x_0, A)\leqslant R}\frac{\ln^{(n_0)}(n_{\i,x_0}(\varepsilon_m))}{\ln\left( \frac{1}{\varepsilon_m} \right)} \geqslant
        \inf\limits_{d(x_0, A)\leqslant m}\frac{\ln^{(m)}(n_{\i,x_0}(\varepsilon_m))}{\ln\left( \frac{1}{\varepsilon_m} \right)} \geqslant m \to \infty.
    \]
    Hence $\i\in\I_\infty$.

    Now we continue the proof of Corollary \ref{cor: I infty is residual}. For all $n\in\N$, define $\psi_n:(0,\infty)\longrightarrow(0,\infty)$ with a formula
    \[
        \psi_n(\varepsilon)= \exp^{(n-1)}\left( \frac{1}{\varepsilon} \right),
    \]
    where $\exp^0=\text{id}$. Clearly $\psi_n(t)<\psi_m(t)$ if $n<m$ and $t>0$. Hence if the assumptions of Theorem \ref{thm:MAIN} are satisfied for some $\psi_{n_0}$, then they are satisfied for any $\psi_n,\ n\geqslant n_0$. We will now prove that there exists an increasing sequence $(n_l)\subseteq \N$ such that
    \begin{equation*}\label{eq: thm MAIN assumptions}
        \lim\limits_{l\to\infty} n_l \frac{N(C_{n_l})}{\psi_2(3C_{n_l})} = 0.
    \end{equation*}
    By Lemma \ref{lem: lower box dimension sequentially}, we know that there exists an increasing sequence $(n_l)\subseteq\N$ such that
    \begin{equation*}\label{eq: lower box dim}
        z_0=\lim\limits_{l\to\infty} \frac{\ln(N(C_{n_l}))}{\ln\left( \frac{1}{C_{n_l}} \right)},
    \end{equation*}
    where $z_0:= \underline{D}(A)$. Hence, we can find $l_0\in\N$ such that for any $l\geqslant l_0$ we have
    \begin{equation}\label{eq: N Cm bound}
        N(C_{n_l})\leqslant \left( \frac{1}{C_{n_l}} \right)^{z_0+1}.
    \end{equation}
    Next find $l_1\in\N$ such that for any $l\geqslant l_1$, it holds
    \begin{equation}\label{eq: nl bound}
        n_l\leqslant \left( \frac{1}{L^{n_l}} \right). 
    \end{equation}
    Let $B:=\diam A +1$. Then for $l\geqslant\max\{l_0,l_1\}$, by (\ref{eq: N Cm bound}) and (\ref{eq: nl bound}), we have
    \begin{align*}
        n_l \frac{N(C_{n_l})}{\psi_2(3C_{n_l})} \leqslant \frac{\left( \frac{1}{L^{n_l}} \right)\cdot \left( \frac{1}{L^{n_l} B} \right)^{z_0+1}}{(\exp\big( \frac{1}{3B} \big))^{\left( \frac{1}{L^{n_l}} \right)}}\leqslant \left( \frac{1}{B} \right)^{z_0 + 1} \frac{\left(\frac{1}{L^{n_l}}\right)^{z_0+2}}{(\exp\big( \frac{1}{3B} \big))^{\left( \frac{1}{L^{n_l}} \right)}}\to 0,
    \end{align*}
    since we have that $\lim\limits_{x\to\infty}\frac{x^b}{a^x}=0$ for any $a>1,b>0$.

    To prove the result we will show that
    \[
        \bigcap\limits_{n\geqslant 2} \I_{\psi_n} \subseteq \I_\infty' =\I_\infty.
    \]
    Since by Theorem \ref{thm:MAIN} all $\I_{\psi_n}$ for $n\geqslant 2$ are residual, the set $\I_\infty$ is also residual.

    Let $\i\in \bigcap_{n\in\N} \I_{\psi_n}$ and fix $n\geqslant 2$. Since $\i\in \I_{\psi_n}$, there exists $(\varepsilon_k)$ such that for every $R>0$,
    \begin{equation}\label{eq: i in I psi n}
        \sup_{d(x_0,A)\leqslant R}\left|\frac{n_{\i,x_0}(\varepsilon_k)}{\psi_n(\varepsilon_k)}-1\right|\to0.
    \end{equation}
     Fix $R>0$ and fix any $0<\delta<1$. Becasue $\psi_n(\varepsilon_k)\to\infty$, as $k\to\infty$, there exists $k_1\in\N$ such that for all $k\geqslant k_1$, $\psi_n(\varepsilon_k)(1-\delta)>1$. Next, we can find $k_2\in\N,\ k_2\geqslant k_1$ such that for any $k\geqslant k_2$ we have $\ln(\psi_n(\varepsilon_k)(1-\delta))>1$. Continuing the process we can find $k_n\in \N$ such that for all $k\geqslant k_n$
    \begin{equation}\label{eq: ln well defined}
        \ln^{(i)}(\psi_n(\varepsilon_k)(1-\delta))>1 \text{ for }i=0,1,\dots,n-1.
    \end{equation}
    Now, using (\ref{eq: i in I psi n}), we find $k_0\geqslant k_n$ such that for all $k\geqslant k_0$ and for all $x_0\in X$ with $d(x_0, A)\leqslant R$, we have
    \[
        1-\delta \leqslant \frac{n_{\i, x_0}(\varepsilon_k)}{\psi_n(\varepsilon_k)}\leqslant 1 + \delta,
    \]
    and in consequence
    \begin{equation}\label{eq: limit definition}
        \psi_n(\varepsilon_k)(1-\delta) \leqslant n_{\i, x_0}(\varepsilon_k)\leqslant \psi_n(\varepsilon_k)(1 + \delta).
    \end{equation}
    Now we apply the natural logarithm to (\ref{eq: limit definition}) $n$ times using (\ref{eq: ln well defined}), and we get that for $k\geqslant k_0$,
    \[
        \ln^{(n)}(\psi_n(\varepsilon_k)(1-\delta)) \leqslant \inf\limits_{d(x_0, A)\leqslant R}\ln^{(n)}(n_{\i, x_0}(\varepsilon_k)) \leqslant \sup\limits_{d(x_0, A)\leqslant R}\ln^{(n)}(n_{\i, x_0}(\varepsilon_k)) \leqslant \ln^{(n)}(\psi_n(\varepsilon_k)(1 + \delta)).
    \]
    \[
        \ln^{(n-1)}(\ln(\psi_n(\varepsilon_k)) + \ln(1-\delta)) \leqslant \inf\limits_{d(x_0, A)\leqslant R}\ln^{(n)}(n_{\i, x_0}(\varepsilon_k))\leqslant \sup\limits_{d(x_0, A)\leqslant R}\ln^{(n)}(n_{\i, x_0}(\varepsilon_k)) \leqslant \ln^{(n-1)}(\ln(\psi_n(\varepsilon_k)) + \ln(1+\delta))
    \]
    \[
        \frac{\ln^{(n-1)}(\ln(\psi_n(\varepsilon_k)) + \ln(1-\delta))}{\ln^{(n)}(\psi_n(\varepsilon_k))}\leqslant \inf\limits_{d(x_0, A)\leqslant R}\frac{\ln^{(n)}(n_{\i, x_0}(\varepsilon_k))}{\ln\left( \frac{1}{\varepsilon_k} \right)} \leqslant \sup\limits_{d(x_0, A)\leqslant R}\frac{\ln^{(n)}(n_{\i, x_0}(\varepsilon_k))}{\ln\left( \frac{1}{\varepsilon_k} \right)} \leqslant \frac{\ln^{(n-1)}(\ln(\psi_n(\varepsilon_k)) + \ln(1+\delta))}{\ln^{(n)}(\psi_n(\varepsilon_k))}
    \]
    \begin{equation*}\label{eq: logarithmic convergence}
        \frac{\ln^{(n-1)}(\ln(\psi_n(\varepsilon_k)) + \ln(1-\delta))}{\ln^{(n-1)}(\ln(\psi_n(\varepsilon_k)))}\leqslant \inf\limits_{d(x_0, A)\leqslant R}\frac{\ln^{(n)}(n_{\i, x_0}(\varepsilon_k))}{\ln\left( \frac{1}{\varepsilon_k} \right)} \leqslant \sup\limits_{d(x_0, A)\leqslant R}\frac{\ln^{(n)}(n_{\i, x_0}(\varepsilon_k))}{\ln\left( \frac{1}{\varepsilon_k} \right)} \leqslant \frac{\ln^{(n-1)}(\ln(\psi_n(\varepsilon_k)) + \ln(1+\delta))}{\ln^{(n-1)}(\ln(\psi_n(\varepsilon_k)))}.
    \end{equation*}
    Now observe that
    \[
        \frac{\ln(\psi_n(\varepsilon_k)) + \ln(1\pm \delta)}{\ln(\psi_n(\varepsilon_k))} = 1 + \frac{\ln(1\pm \delta)}{\ln(\psi_n(\varepsilon_k))} \to 1\text{, with }k\to\infty,
    \]
    hence applying Lemma \ref{lem: convergence of logarithms} $n-1$ times, we get that
    \[
        \limsup\limits_{k\to\infty} \bigg(\sup\limits_{d(x_0, A)\leqslant R} \frac{\ln^{(n)}(n_{\i, x_0}(\varepsilon_k))}{\ln\left( \frac{1}{\varepsilon_k} \right)}\bigg)\leqslant 1,
    \]
    and
    \[
        \liminf\limits_{k\to\infty} \bigg(\inf\limits_{d(x_0, A)\leqslant R} \frac{\ln^{(n)}(n_{\i, x_0}(\varepsilon_k))}{\ln\left( \frac{1}{\varepsilon_k} \right)}\bigg)\geqslant 1.
    \]
    By Lemma \ref{lem: limsup limlinf equiv}, we obtained
    \[
        \lim\limits_{k\to\infty} \sup_{d(x_0, A)\leqslant R} \left|\frac{\ln^{(n)}(n_{\i, x_0}(\varepsilon_k))}{\ln\left( \frac{1}{\varepsilon_k} \right)} - 1\right| = 0,
    \]
    which gives
    \begin{equation}\label{eq: logarithmic condition}
        \forall_{0<\eta<1}\ \exists_{k_0\in\N}\ \forall_{k\geqslant k_0}\ \forall_{d(x_0, A)\leqslant R} \left| \frac{\ln^{(n)}(n_{\i, x_0}(\varepsilon_k))}{\ln\left( \frac{1}{\varepsilon_k} \right)}  -1\right| \leqslant \eta.
    \end{equation}
    Now fix $M>1$. First choose $k_0$ according to (\ref{eq: logarithmic condition}) for $\eta=\frac{M-1}{M}$.
    
    Next, find $t_0\in\N$ such that for all $t\geqslant t_0$, we have
    \[
        \frac{\ln^{(n-1)}(t)}{\ln(\ln^{(n-1)}(t))}\geqslant M^2.
    \]
    Since $N(\varepsilon_k)\to\infty$ with $k\to\infty$, we can find $k_1\in\N$ such that $N(\varepsilon_k)\geqslant t_0+1$ for $k\geqslant k_1$. Now for any $k\geqslant \max\{k_0,k_1\}$ and $x_0\in X$ with $d(x_0, A)\leqslant R$, it holds
    \[
        n_{\i, x_0}(\varepsilon_k)+1 \geqslant N(\varepsilon_k) \geqslant t_0+1,
    \]
    thus
    \[
        \frac{\ln^{(n-1)}(n_{\i, x_0}(\varepsilon_k))}{\ln(\ln^{(n-1)}(n_{\i, x_0}(\varepsilon_k)))}\geqslant M^2.
    \]
    and in consequence
    $$
        \frac{\ln^{(n-1)}(n_{\i, x_0}(\varepsilon_k))}{\ln\left( \frac{1}{\varepsilon_k} \right)} = \frac{\ln^{(n-1)}(n_{\i, x_0}(\varepsilon_k))}{\ln^{(n)}(n_{\i, x_0}(\varepsilon_k))} \cdot \frac{\ln^{(n)}(n_{\i, x_0}(\varepsilon_k))}{\ln\left( \frac{1}{\varepsilon_k} \right)}\\
        \geqslant M^2(1-\delta)=M^2 \cdot \frac{1}{M} = M.
    $$
    In conclusion we obtained the condition
    \[
    \forall_{M>1}\ \exists_{k_0\in\N}\ \forall_{k\geqslant k_0}\ \forall_{d(x_0,A)\leqslant R}\ \frac{\ln^{(n-1)}(n_{\i, x_0}(\varepsilon_k))}{\ln\left( \frac{1}{\varepsilon_k} \right)} \geqslant M.
    \]
    Therefore
    \[
        \inf\limits_{d(x_0, A)\leqslant R}\frac{\ln^{(n-1)}(n_{\i,x_0}(\varepsilon_k))}{\ln\left( \frac{1}{\varepsilon_k} \right)}\to \infty
    \]
    Since $n\geqslant 2$ was arbitrary we complete the proof.
\end{proof}

Now we prove that the main result from \cite{LSS2} can be inferred from Theorem \ref{thm:MAIN}.

\begin{corollary}\label{cor: MAIN implies LSS2}
    Theorem \ref{thm: LSS2 MAIN} follows from Theorem \ref{thm:MAIN}
\end{corollary}
\begin{proof}
    According to Remark \ref{rem: LSS2 MAIN}, we only have to prove that for any $z\in(\underline{D}(A),\infty)$, the set
    \[
        \I_z:=\{ \i\in I^\infty:\ \exists_{(\varepsilon_k)\in c_0^+}\ \forall_{R>0}\ \sup_{d(x_0,A)\leqslant R}\bigg|\frac{\ln(n_{\i,x_0}(\varepsilon_k))}{\ln\left( \frac{1}{\varepsilon_k} \right)}-z\bigg|\to0 \}.
    \]
is residual.\\
Fix $z\in (\underline{D}(A),\infty)$.
We will show that $\I_{\psi_z}\subseteq \I_z$, where $\psi_z$ is given by the formula
    \[
        \psi_z(\varepsilon)=\left( \frac{1}{\varepsilon} \right)^z,\text{ for }\varepsilon\in(0,\infty),
    \]
and that $\I_{\psi_z}$ is residual.\\
    First, observe that 
    \[
        \liminf_{n\to\infty} n \frac{N(C_{n})}{\psi_z(3C_{n})} = 0,
    \]
    and hence $\I_{\psi_z}$ is residual by Theorem \ref{thm:MAIN}.

    Indeed, by Lemma \ref{lem: lower box dimension sequentially}, there exists an increasing sequence $(n_l)\subseteq\N$ for which
    \[
        z_0=\lim\limits_{l\to\infty} \frac{\ln(N(C_{n_l}))}{\ln\left( \frac{1}{C_{n_l}} \right)},
    \]
    hence there exists $l_0\in\N$ such that for any $l\geqslant l_0$, we have
    \[
        N(C_{n_l})\leqslant \left( \frac{1}{C_{n_l}} \right)^{\frac{z+z_0}{2}}.
    \]
    Let $B:=\diam A+1$. Then, for any $l\geqslant l_0$, we get that
    \[
        n_l \frac{N(C_{n_l})}{\psi_z(3C_{n_l})}\leqslant n_l \frac{\left( \frac{1}{C_{n_l}} \right)^{\frac{z+z_0}{2}}}{\left( \frac{1}{3} \right)^z\left(\frac{1}{C_{n_l}}  \right)^z} = 3^zn_l C_{n_l}^{\frac{z-z_0}{2}} = 3^zB^{\frac{z-z_0}{2}}n_lL^{n_l\left(\frac{z-z_0}{2}\right)}\to0
    \]
    
    Continuing the proof of Corollary \ref{cor: MAIN implies LSS2}, let $\i\in \I_{\psi_z}$. By definition of $\I_{\psi_z}$ we have a sequence $(\varepsilon_k)$ such that for every $R>0$,
    \begin{equation}\label{eq: def of Iz}
        \sup_{d(x_0,A)\leqslant R}\left|\frac{n_{\i,x_0}(\varepsilon_k)}{\psi_z(\varepsilon_k)}-1\right|\to0.
    \end{equation}
    Without loss of generality, we can assume that $\ve_k<1$ for all $k\in\N$.
    Fix $R>0$ and fix any $0<\delta<1$. Then using (\ref{eq: def of Iz}), there exists $k_0\in\N$ such that for every $k\geqslant k_0$ and $x_0\in X$ with $d(x_0, A)\leqslant R$, we have
    \[
        \psi_z(\varepsilon_k)(1-\delta)\leqslant n_{\i, x_0}(\varepsilon_k) \leqslant \psi_z(\varepsilon_k)(1+\delta).
    \]
    Applying the natural logarithm to the above inequality, we get
    \[
        z\ln\left( \frac{1}{\varepsilon_k} \right) + \ln(1-\delta) \leqslant \inf\limits_{d(x_0,A)\leqslant R}\ln(n_{\i, x_0}(\varepsilon_k))\leqslant \sup\limits_{d(x_0,A)\leqslant R}\ln(n_{\i, x_0}(\varepsilon_k)) \leqslant z\ln\left( \frac{1}{\varepsilon_k} \right) + \ln(1+\delta)
    \]
    and hence for $k\geq k_0$,
    \begin{equation}\label{eq: z logarithms}
        z + \frac{\ln(1-\delta)}{\ln\left( \frac{1}{\varepsilon_k} \right)} \leqslant \inf\limits_{d(x_0,A)\leqslant R}\frac{\ln(n_{\i, x_0}(\varepsilon_k))}{\ln\left( \frac{1}{\varepsilon_k} \right)}\leqslant \sup\limits_{d(x_0,A)\leqslant R}\frac{\ln(n_{\i, x_0}(\varepsilon_k))}{\ln\left( \frac{1}{\varepsilon_k} \right)} \leqslant z + \frac{\ln(1+\delta)}{\ln\left( \frac{1}{\varepsilon_k} \right)}.
    \end{equation}
    Since
    \[
        \lim\limits_{k\to\infty}\frac{\ln(1\pm\delta)}{\ln\left( \frac{1}{\varepsilon_k} \right)}= 0,
    \]
    we get that
    \[
        \limsup\limits_{k\to\infty} \left(\sup\limits_{d(x_0,A)\leqslant R} \frac{\ln(n_{\i, x_0}(\varepsilon_k))}{\ln\left( \frac{1}{\varepsilon_k} \right)}\right)\leqslant z,
    \]
    and
    \[
        \liminf\limits_{k\to\infty} \left(\inf\limits_{d(x_0,A)\leqslant R} \frac{\ln(n_{\i, x_0}(\varepsilon_k))}{\ln\left( \frac{1}{\varepsilon_k} \right)}\right)\geqslant z.
    \]
    Hence, by Lemma \ref{lem: limsup limlinf equiv}, we obtain that
    \[
        \sup_{d(x_0,A)\leqslant R}\left|\frac{\ln(n_{\i,x_0}(\varepsilon_k))}{\ln\left( \frac{1}{\varepsilon_k} \right)}-z\right|\to0.
    \]
\end{proof}
Now, let us look at an example showing that $\I_z$ and $\I_{\psi_z}$, defined as in the proof of Corollary \ref{cor: MAIN implies LSS2}, are, in general, not the same sets.

\begin{example}\label{ex:ext}
    Let $X:=\R$ and let $\F:=\{f_1,f_2\}$, where
    \[
        f_1(x):=\frac{1}{2} x,\quad f_2(x):=1,\text{ for }x\in\R.
    \]
    Then,
    \[
        A=\{0\}\cup \{(\tfrac{1}{2^n}):\ n\in\N\cup \{0\}\}
    \]
    is the attractor of $\F$, and $\underline{D}(A)=0$.

    We will show that for any $z>0=\underline{D}(A)$, there exists a sequence $\i\in\{1,2\}^\infty$ such that $\i\in \I_z$ but $\i\notin\I_{\psi_z}$ for $\psi_z(\varepsilon)=\left( \frac{1}{\varepsilon}\right)^z$.

    Fix $z>0$ and choose $k_1\in\N$ such that $k<2^{kz}$ for all $k\geqslant k_1$. Then choose $k_2\in\N$ such that $(k+1)2^{kz}>\frac{1}{2^z-1}$ for all $k\geqslant k_2$. Then, for any $k\geqslant k_0:=\max\{k_1,k_2\}$, we have
    \[
        \lfloor k\cdot 2^{kz} \rfloor + k\leqslant k\cdot 2^{kz} + k = (k+1) 2^{kz} <(k+1)2^{(k+1)z}-1<\lfloor (k+1)\cdot 2^{(k+1)z} \rfloor.
    \]
    Now define $\i$ in the following way: let $i_n:=1$, for $n=\lfloor k\cdot 2^{kz} \rfloor,\dots, \lfloor k\cdot 2^{kz} \rfloor + k -1$, where $ k\geqslant 2k_0$, and $i_n:= 2$ for other $n$s.

    Define $\varepsilon_k:=(\tfrac{1}{2})^k$, for $k\in\N$. Then $\ln(\tfrac{1}{\varepsilon_k})=k\ln(2)$. We continue the proof by providing a detailed analysis for $n_{\i,x_0}(\ve_k)$ for arbitrary $x_0\in X$. First, observe that the first block of $k\geqslant 2k_0$ symbols 1 starts in $\i$ at the position $n=\lfloor k\cdot 2^{kz} \rfloor$, so for any $k\geqslant k_0$ and every $x_0\in X$ we have
    \[
        n_{\i, x_0}(\varepsilon_k) \leqslant \lfloor k\cdot 2^{kz} \rfloor + k -1.
    \]
    Now, consider the following cases: $x_0\leqslant0$; $x_0\in(0,1)$; $x_0\geqslant1$.

    Let $x_0\leqslant 0$. Since $i_1=2$ we clearly have that 
    \[
        n_{\i, x_0}(\ve_k)\geqslant n_{\i, 0}(\ve_k) = \lfloor (k-1)\cdot 2^{(k-1)z} \rfloor + k -2.
    \]
    The equality follows from the fact that the closed ball $B(0, \ve_k)$ covers the set $\{0\}\cup \{\tfrac{1}{2^n}:\ n\geqslant k\}$.

    When $x_0\in(0,1)$ find $l\in \N$ such that $\tfrac{1}{2^l}\leqslant x_0 < \tfrac{1}{2^{l-1}}$. Now, for $k\geqslant 2k_0$ define $r(l,k)$ by the formula
    \[
        r(l, k)=\left\{ \begin{aligned}
            k-1,&\ l>k\\
            k-2,&\ l=k\\
            k,&\ l<k
        \end{aligned}\right.,
    \]
    then we have the following
    \[
        n_{\i, x_0}(\ve_k)\geqslant\lfloor r(l,k)\cdot 2^{r(l,k)z}\rfloor + r(l,k)-1.
    \]
    Clearly, if $x_0\geqslant 1$ we have that for any $k\geqslant 2k_0$ the following holds $n_{\i, x_0}(\varepsilon_k) = \lfloor k\cdot 2^{kz} \rfloor + k -1$.
    Using the above calculations we get that for sufficiently large $k\geqslant 2k_0$, we have
    \[
        \sup\limits_{x_0\in X} \frac{\ln(n_{\i, x_0}(\ve_k))}{\ln(\tfrac{1}{\varepsilon_k})} = \frac{\ln(\lfloor k\cdot 2^{kz} \rfloor + k -1)}{\ln(\tfrac{1}{\varepsilon_k})}\leqslant \frac{\ln(k\cdot 2^{kz}+ k -1)}{k\ln(2)} \leqslant \frac{\ln(k+1)}{k\ln(2)} + \frac{kz\ln(2)}{k\ln(2)}\leqslant z+ \frac{\ln(k+1)}{k\ln(2)},
    \]
    and
    \[
        \inf\limits_{x_0\in X} \frac{\ln(n_{\i, x_0}(\ve_k))}{\ln(\tfrac{1}{\varepsilon_k})} = \frac{\ln(\lfloor (k-2)\cdot 2^{(k-2)z} \rfloor + k -3)}{\ln(\tfrac{1}{\varepsilon_k})} \geqslant \frac{\ln( (k-2)\cdot 2^{(k-2)z})}{\ln(\tfrac{1}{\varepsilon_k})} \geqslant \frac{\ln(k-2)}{k\ln(2)} + \frac{(k-2)z\ln(2)}{k\ln(2)}.
    \]
    Since $\lim\limits_{k\to\infty}\frac{\ln(k-2)}{k}=\lim\limits_{k\to\infty}\frac{\ln(k+1)}{k}=0$ and $\lim\limits_{k\to\infty} \frac{k-2}{k}=1$, using Lemma \ref{lem: limsup limlinf equiv} we get that $\i\in\I_z$.

    Next, consider any sequence $(\varepsilon_k')\in c_0^+$. For every $k\in\N$, let $m(k)\in\N$ be such that $\varepsilon_{m(k)+1}<\varepsilon_k' \leqslant \varepsilon_{m(k)}$. Then for sufficiently large $k\in\N$,
    \[
        \frac{n_{\i, x_0}(\varepsilon_k')}{(\tfrac{1}{\varepsilon_k'})^z}\geqslant \frac{n_{\i, x_0}(\varepsilon_{m(k)})}{(\tfrac{1}{\varepsilon_{m(k)+1}})^z} \geqslant \frac{ (m(k)-2)\cdot 2^{(m(k)-2)z} }{2^{(m(k)+1)z}} \to \infty.
    \]
    In conlusion $\i\notin \I_{\psi_z}$. In fact, we have shown a stronger condition, since not only for any $R>0$,
    \[
        \sup\limits_{d(x_0, A)\leqslant R} \left| \frac{n_{\i, x_0}(\varepsilon_k)}{\psi_z(\varepsilon_k)} -1 \right|\nrightarrow 0,
    \]
    but
    \[
        \inf\limits_{x_0\in X} \frac{n_{\i, x_0}(\varepsilon_k)}{\psi_z(\varepsilon_k)} \to \infty.
    \]
\end{example}

One can ask a question, what happens to sets $\I_\psi$ and $\I_\varphi$, from Theorem \ref{thm:MAIN}, for different functions $\psi$ and $\varphi$. From now on we restrict to functions tending to $\infty$ when $\ve\to 0$.

\begin{lemma}\label{lem: finite distance implies I phi equal I psi}
Let $\psi,\varphi:(0,\infty)\longrightarrow(0,\infty)$ be such that $\psi(\varepsilon),\varphi(\varepsilon)\to\infty$ with $\varepsilon\to 0$. If
    \[
        \frac{\varphi(\varepsilon)}{\psi(\varepsilon)}\to 1\text{ as }\varepsilon\to 0,
    \]
    then $\I_\psi = \I_\varphi$.
\end{lemma}
\begin{proof} Let $\i\in \I\varphi$. Then there exists $(\varepsilon_k)\in c_0^+$ such that for every $R>0$ we have 
    \begin{equation}\label{eq: i in I phi}
        \sup_{d(x_0,A)\leqslant R}\left|\frac{n_{\i,x_0}(\varepsilon_k)}{\varphi(\varepsilon_k)}-1\right|\to0.
    \end{equation}
    Fix $\delta>0$ and $R>0$. By (\ref{eq: i in I phi}), we can find $k_0\in\N$ such that for all $k\geqslant k_0$ and all $x_0\in X$ with $d(x_0, A)\leqslant R$, we have
    \[
        1-\delta\leqslant \frac{n_{\i,x_0}(\varepsilon_k)}{\varphi(\varepsilon_k)} \leqslant 1+\delta.
    \]
     Now since
    \[
        \frac{\varphi(\varepsilon_k)}{\psi(\varepsilon_k)}\to 1 \text{ as }k\to \infty,
    \]
    we can find $k_1\in\N$ such that for any $k\geqslant k_1$, we have
    \[
        \frac{1-2\delta}{1-\delta}\leqslant \frac{\varphi(\varepsilon_k)}{\psi(\varepsilon_k)} \leqslant \frac{1+2\delta}{1+\delta}.
    \]
   Since
    \[
        \frac{n_{\i,x_0}(\varepsilon_k)}{\psi(\varepsilon_k)} = \frac{n_{\i,x_0}(\varepsilon_k)}{\varphi(\varepsilon_k)}\cdot\frac{\varphi(\varepsilon_k)}{\psi(\varepsilon_k)},
    \]
    for any $k\geqslant \max\{k_0,k_1\}$ and for all $x_0\in X$ with $d(x_0, A)\leqslant R$, we have
    \[
        1-2\delta\leqslant\frac{n_{\i,x_0}(\varepsilon_k)}{\psi(\varepsilon_k)}\leqslant 1+2\delta.
    \]
    and thus
  \[
        \left|\frac{n_{\i,x_0}(\varepsilon_k)}{\psi(\varepsilon_k)}-1\right|\leq 2\delta.
    \]
    Hence, we showed that
    \[
        \sup_{d(x_0,A)\leqslant R}\left|\frac{n_{\i,x_0}(\varepsilon_k)}{\psi(\varepsilon_k)}-1\right|\to0,
    \]
    so $\i \in \I_\varphi$. Repeating the argument with $\psi$ and $\varphi$ swapped completes the proof.
\end{proof}
\begin{remark}\emph{
    Note that, if $\psi,\varphi:(0,\infty)\to(0,\infty)$ satisfy $\psi(\ve),\varphi(\ve)\to\infty$ for $\ve\to 0$, and
    \[
        d_\infty(\psi, \varphi) = \lVert \psi - \varphi \rVert_\infty := \sup_{\varepsilon>0} |\psi(\varepsilon) - \varphi(\varepsilon)| < \infty,
    \]
    then
    \[
        \frac{\varphi(\varepsilon)}{\psi(\varepsilon)}\to 1,\text{ as }\varepsilon\to 0.
    \]
In particular, if $d_\infty(\psi, \varphi)<\infty$, then $\I_\psi=\I_\varphi$.}
\end{remark}

The following corollary extends somehow Theorem \ref{thm:MAIN} for countable families of functions $\varphi$.
\begin{corollary}
    Let $\bm\psi=\{\psi_k: k\in\N\}$, where for $k\in\N$, $\psi_k:(0,\infty)\longrightarrow(0,\infty)$ is a function such that $\lim\limits_{\varepsilon\to 0} \psi_k(\varepsilon)=\infty$ and additionally, for each $k\in\N$
    \[
        \liminf_{n\to\infty} n \frac{N(C_{n})}{\psi_k(3C_{n})}=0.
    \]
    Then the set
    \[
        \I^{\bm\psi} :=\bigcap_{\psi\in \bm\psi^+} \I_\psi, \text{ where }\bm\psi^+:=\{\psi:\ \exists_{n\in\N}\ \lim\limits_{\varepsilon\to 0} \frac{\psi(\varepsilon)}{\psi_n(\varepsilon)}=1\},
    \]
    is residual in $I^\infty$.
\end{corollary}
\begin{proof}
It is enough to show that 
    \[
        \I^{\bm\psi}= \bigcap_{k\in\N}\I_{\psi_k} \]

    Clearly $\I^{\bm\psi} \subseteq \bigcap_{k\in\N} \I_{\psi_k}$. To see the converse one, note that if $\psi\in \bm\psi^+$, then
    there exists $\psi_k$ such that $\lim\limits_{\varepsilon\to 0} \frac{\psi(\varepsilon)}{\psi_k(\varepsilon)}=1$ and by Lemma \ref{lem: finite distance implies I phi equal I psi}, we have that $\I_{\psi_k}=\I_\psi$.

\end{proof}

\section{Champernowne and de Bruijn sequences}

In this section we study some properties of drivers that ensure "fast" convergence of the chaos game. The aim is to obtain a result analogous to Theorem \ref{thm: as Champ fast conv}. Again, we assume that $\F=\{f_i:i\in I\}$, where $I=\{1,...,K\}$, is a Banach contractive IFS on a complete metric space $X$ and $A$ is its attractor. For this section, we assume that $K\geqslant 2$.
\begin{example}(a Champernowne sequence)\\
First let us consider a classical Champernowne driver $\i$ (see \cite{Cal}), that is, a sequence in which we first write all words of length $1$, then all words of length $2$ and so on. Then for any $m\in\N$,
$$
\mathbf{n}_\i(m)\leqslant K+2K^2+...+mK^m=\frac{K-K^{m+1}(m+1)+mK^{m+2}}{(K-1)^2}
$$
In particular,
$$
\limsup_{m\to\infty}\frac{\mathbf{n}_\i(m)}{mK^m}\leqslant\frac{K}{K-1}
$$
Using \cite[Lem. 15(b)]{LSS1}, for any $x_0\in X$, we get

\begin{equation}\label{eq:champerr}
n_{\i,x_0}(c_{m}(x_0))\leqslant \mathbf{n}_\i(m)
\end{equation}where $c_m(x_0):=L^m(\diam A+d(x_0,A))$. Thus
\begin{equation}\label{eq:champer}
\limsup_{m\to\infty}\big(\sup_{x_0\in X}\frac{n_{\i,x_0}(c_{m}(x_0))}{mK^m}\big)\leqslant \limsup_{m\to\infty}\frac{\mathbf{n}_\i(m)}{mK^m} \leqslant\frac{K}{K-1},
\end{equation}

This means that for any  $C>\frac{K}{K-1}$, there exists $m_0\in\N$ such that for every $x_0\in X$ and $m\geq m_0$, 
\begin{equation}\label{eq:champer2}
n_{\i,x_0}(c_{m}(x_0))\leqslant CmK^m.
\end{equation}
\end{example}
A natural question arises whether we can have fastest convergence as in (\ref{eq:champer}) and in consequence in (\ref{eq:champer2}).\\
Recall that a finite sequence $\i$ is called \emph{noncyclic de Bruijn sequence of order $m$} (see \cite{All}), if every finite word of length $m$ appears in $\i$ exactly once. This ensures an optimal "packing" of finite words. The length of a noncyclic de Bruijn sequence of order $m$ is exactly $K^m+m-1$. The following result from \cite{Bech} show the possibility of extending de Bruijn sequences (note that in item (2) we cannot replace $m+2$ by $m+1$).
\begin{theorem}\label{thm:deBrui}\cite[Thm. 1]{Bech}$\;$\\
(1) If $K\geq 3$, then every de Bruijn sequence of order $m$ can be extended to a de Bruijn sequence of order $m+1$.\\
(2) If $K=2$, then every de Bruijn sequence of order $m$ can be extended to a de Bruijn sequence of order $m+2$.
\end{theorem}

\begin{example}(an infinite de Bruijn sequence).\\
Motivated by Theorem \ref{thm:deBrui}, in \cite{Bech} there was introduced the following definition (\cite[Def. 6]{Bech}):\\
(1) If $K\geq 3$, then an infinite driver $\i$ is called as \emph{an infinite de Bruijn sequence}, if it is the inductive limit of extending de Bruijn sequences of order $m$,
for each $m$.\\
(2) If $K=2$, then an infinite driver $\i$ is called as \emph{an infinite de Bruijn sequence}, if it is the inductive limit of extending de Bruijn sequences of order $2m$, for each $m$.\\
Let $\i$ be an infinite deBruijn sequence. Observe that for every $m\in\N$,\\
if $K\geq 3$, then
$$
\mathbf{n}_\i(m)\leqslant K^m+m-1
$$
and if $K=2$, then
$$
\mathbf{n}_\i(m)\leqslant \left\{\begin{array}{ccc}K^m+m-1&\mbox{if}&m\;\mbox{is even}\\
K^{m+1}+m&\mbox{if}&m\;\mbox{is odd}\end{array}\right.
$$
For simplicity of notation, set 
$$\alpha(K):=\left\{\begin{array}{ccc}1&\mbox{if}&K\geq 3\\2&\mbox{if}&K=2\end{array}\right.$$
Again using (\ref{eq:champerr}) and the above estimations, we have
\begin{equation}\label{eq:debrui}
\limsup_{m\to\infty}\big(\sup_{x_0\in X}\frac{n_{\i,x_0}(c_{m}(x_0))}{K^m}\big)\leqslant \limsup_{m\to\infty}\frac{\mathbf{n}_\i(m)}{K^m} \leqslant \limsup_{m\to\infty}\frac{\mathbf{n}_\i(m)}{K^m}\leqslant\alpha(K),
\end{equation}
where again $c_m(x_0):=L^m(\diam A+d(x_0,A))$. In fact, for $K\geq 3$ we have a limit above: $\lim_{m\to\infty}\frac{n_{\i,x_0}(c_m(x_0))}{K^m}=1$, as nothing below $1$ can be reached.\\
Condition (\ref{eq:debrui}) implies that for any  $C>\alpha(K)$, there exists $m_0\in\N$ such that for every $x_0\in X$ and $m\geq m_0$, 
\begin{equation}\label{eq:debrui2}
n_{\i,x_0}(c_{m}(x_0))\leqslant CK^m.
\end{equation}
Observe that such drivers give better rate of convergence than a Chambernowne sequence as (\ref{eq:debrui}) implies
$$
\lim_{m\to\infty}\frac{n_{\i,x_0}(c_{m}(x_0))}{mK^m}=0.
$$

Now we give some insight how infinite de Bruijn drivers behave from the perspective of Theorem \ref{thm:MAIN}. Namely, for $x_0\in X$, we find a function $\psi_{x_0}$, which bounds $n_{\i,x_0}(\ve)$ in the limit $\ve\to 0$.

Fix any $x_0\in X$. Define $\psi_{x_0}(\ve):(0,\infty)\longrightarrow (0,\infty)$ with the formula
\[
    \psi_{x_0}(\ve):=K\cdot \alpha(K)\cdot (\diam A+d(x_0,A))^{\log_L(K)}\cdot \bigg(\frac{1}{\ve}\bigg)^{\log_L(K)},\ \ve>0.
\]
Note that for any $m\in\N$, we get that
\[
    \psi_{x_0}(c_{m-1}(x_0)) = K^{m}\cdot \alpha(K).
\]
Let $\ve>0$. Find $m\in \N$ such that $c_m\leqslant\ve< c_{m-1}$. Then by (\ref{eq:debrui}), we have
\[
    \limsup_{\ve\to 0}\frac{n_{\i,x_0}(\ve)}{\psi_{x_0}(\ve)}\leqslant \limsup_{m\to \infty} \frac{n_{\i, x_0}(c_m(x_0))}{\psi_{x_0}(c_{m-1}(x_0))}\leqslant \limsup_{m\to \infty} \frac{\mathbf{n}_\i(m)}{\psi_{x_0}(c_{m-1}(x_0))} = \limsup_{m\to \infty} \frac{\mathbf{n}_\i(m)}{K^{m}\cdot \alpha(K)} \leqslant 1.
\]

Note that a similar reasoning can be repeated for a Champernowne driver. That is, we can find analogous ,,bounding'' function $\psi$. But its formula would be more complicated, so we omit it for the sake of readability.
\end{example}

\section*{Statements and Declarations}

\subsection*{Author Contributions}
All authors contributed to the study conception and design. All authors read and approved the final manuscript.

\subsection*{Funding}
No funds, grants, or other support was received.

\subsection*{Data Availability}
No datasets were generated or analysed during this study.

\subsection*{Competing Interests}
The authors have no competing interests to declare that are relevant to the content of this article.

\bibliographystyle{amsplain}

\end{document}